\documentclass[12pt]{arxivclassfile}

\def\R{{\mathbb R}}
\def\N{{\mathbb N}}
\def\d{{\rm d}}
\def\e{{\rm e}}

\newcommand{\ft}{\tilde{f}}

\newcommand{\A}{\alpha}
\newcommand{\B}{\beta}

\newcommand{\xhat}{{\hat{x}}}

\newcommand{\be}{\begin{equation}}
\newcommand{\bi}{\begin{itemize}}
\newcommand{\ei}{\end{itemize}}
\newcommand{\ee}{\end{equation}}
\newcommand{\ba}{\begin{eqnarray}}
\newcommand{\bs}{\begin{eqnarray*}}
\newcommand{\ea}{\end{eqnarray}}
\newcommand{\es}{\end{eqnarray*}}
\newcommand{\Om}{\Omega}

\newcommand{\Z}{\mathbb Z}

\usepackage{graphicx,epsfig,psfrag}

\usepackage{mathscinet}
\usepackage{url}
\usepackage{amssymb}
\usepackage{amsmath}
\usepackage{amsthm}
\usepackage{enumerate}
\newtheorem{theorem}{Theorem}[section]

\newtheorem{corollary}{Corollary}[section]
\newtheorem{lemma}{Lemma}[section]
\newtheorem{proposition}{Proposition}[section]

\newtheorem{definition}{Definition}[section]

\def\dist{{\rm dist}}

\begin{document}

\begin{frontmatter}


\author[RL]{R.\ Laister\corref{thing}}
\ead{Robert.Laister@uwe.ac.uk}
\address[RL]{Department of Engineering Design and Mathematics, \\ University of the West of England, Bristol BS16 1QY, UK.}
\author[warwick]{J.C.\ Robinson}
\ead{J.C.Robinson@warwick.ac.uk}
\author[warwick]{M.\ Sier{\.z}\polhk{e}ga}
\ead{M.L.Sierzega@warwick.ac.uk}
\address[warwick]{Mathematics Institute, Zeeman Building,\\ University of Warwick, Coventry CV4 7AL, UK.}
\cortext[thing]{Corresponding author}

\title{Gaussian lower bounds on the Dirichlet heat kernel and non-existence of local solutions for semilinear heat equations of Osgood type}



\begin{abstract}
We give a simple proof of a lower bound for the Dirichlet heat kernel in terms of the Gaussian heat kernel. Using this we establish a non-existence result for  semilinear heat equations with zero Dirichlet boundary conditions and initial data in $L^q(\Omega)$ when the source term $f$ is non-decreasing and $\limsup_{s\to\infty}s^{-\gamma}f(s)=\infty$ for some $\gamma>q(1+2/n)$.
This allows us to construct a  locally Lipschitz $f$  satisfying the Osgood condition $\int_{1}^{\infty}1/f(s)\ \,\d s =\infty$, which ensures global existence for bounded initial data,  such that for every $q$ with $1\le q<\infty$ there is an initial condition $u_0\in L^q(\Om )$ for which the corresponding semilinear problem has no local-in-time solution.

\end{abstract}


\end{frontmatter}


\section{Introduction}
\label{Intro}

In a previous paper \cite{LRS1} we showed that for locally Lipschitz $f$ with $f>0$ on $(0,\infty )$, the Osgood condition
\begin{equation}\label{Osgood1}
\int_1^\infty\frac{1}{f(s)}\,\d s=\infty,
\end{equation}
which ensures global existence of solutions of the scalar ODE $\dot x=f(x)$, is not sufficient to guarantee the local existence of solutions of the Cauchy problem
\begin{equation}\label{inRn}
u_t=\Delta u+f(u)
\end{equation}
for initial data in $L^q(\R^n)$, $1\le q<\infty$. This is in stark contrast to the case of bounded initial data, for which (\ref{Osgood1}) implies
that any solution of (\ref{inRn}) exists  globally in time; see \cite{RS1}, for example.

In \cite{LRS1} we considered the PDE (\ref{inRn}) on the whole space $\R^n$, which allowed us to use in our calculations the explicit form of the Gaussian heat kernel,
\be
G_n(x,y;t)=(4\pi t)^{-n/2}\e^{-|x-y|^2/4t}.\label{eq:Kn}
\ee
 The main result there was that for each $q$ with $1\le q<\infty$ one can find a non-negative,  locally Lipschitz and Osgood $f$ such that there are initial data in $L^q(\R^n)$ for which there is no local-in-time integrable solution of (\ref{inRn}).

In this paper we obtain a similar result for the equation posed with Dirichlet boundary conditions on a bounded domain, by using Gaussian lower bounds on the Dirichlet heat kernel. Indeed, in Section \ref{lbK} (Theorem~\ref{ourvdB}) we give a lower bound for the Dirichlet heat kernel on a bounded domain $\Omega$:
\begin{equation}\label{easyvdb}
K_\Omega(x,y;t)\ge \B^n G_n(x,y;t)\qquad\mbox{for}\quad t\le\epsilon^2/n,
\end{equation}
whenever $[x,y]$, the line segment joining $x$ and $y$, is contained in the interior of $\Omega$ and is always at least a distance $\epsilon$ from the boundary of $\Omega$.  Here, $\B >0$ is an explicit constant. Based on the argument of van den Berg \cite{vdBerg90} we also provide in the appendix a proof of a result valid for all $t>0$
 $$
K_\Omega(x,y;t)\ge \e^{-n^2t/4\epsilon^2}G_n(x,y;t),
$$
but  (\ref{easyvdb}) is sufficient for our purposes and has a significantly simpler proof.


More explicitly, we focus throughout the paper on the following problem (P), posed on a smooth bounded domain $\Om\subset\R^n$:
\bs
({\rm P})\quad\left\{
\begin{array}{rll}
u_t =&  \Delta u +  f(u) & {\rm in}\quad \Om,\\
u = & 0 &{\rm on}\quad \partial\Om,\\
u(x,0) =& u_0(x) &{\rm in}\quad \Om.
\end{array}
\right.
\es
The source term $f:[0,\infty)\to[0,\infty)$ is non-decreasing and satisfies the asymptotic growth condition
\begin{equation}\label{atinf}
\displaystyle\limsup_{s\to\infty}s^{-\gamma}f(s)=\infty.
\end{equation}
We show in Theorem \ref{thm:main} that if (\ref{atinf}) holds for some $\gamma>q(1+2/n)$ then one can find a non-negative $u_0\in L^q(\Om)$ such that there is no solution of (P) that is in $L^1_{\rm loc}(\Omega)$ for $t>0$.

We finish (see Corollary~\ref{cor:lip}) by constructing a
function $f$ that grows quickly enough such that (\ref{atinf}) holds for every $\gamma\ge0$, but nevertheless still verifies the Osgood condition (\ref{Osgood1}). This example shows that there are functions $f$ for which (P) is well posed in $L^\infty(\Omega)$ but not in any $L^p(\Omega)$ with $1\le p<\infty$.

One can see this result as in some sense dual to that of Fila et al. \cite{FNV} (see also Section 19.3 of \cite{SQ}), who show that there exists an $f$ such that all positive solutions of $\dot x=f(x)$ blow up in finite time while all solutions of (P) with Dirichlet boundary conditions are global and bounded.


\section{A Gaussian lower bound for the Dirichlet heat kernel}\label{lbK}

For any smooth domain $D$ in $\R^n$ (i.e.\ $D$ is smooth, open, and connected), we denote by $K_{D}(x,y,t)$  the Dirichlet heat kernel associated with the Dirichlet heat semigroup $S_{D}(t)$, i.e.
\be
w_{D}(x,t)=(S_{D}(t)w_0)(x):=\int_{D }K_{D}(x,y;t) w_0(y)\, \d y
\label{eq:DHK}\ee
is the classical solution of the linear heat equation
\bs
 w_t = \Delta w  &{\rm in}& D,\label{eq:HE}\\
 w = 0 &{\rm on}& \partial D,\label{eq:HBC}\\
 w= w_0 & {\rm in} & D ,\label{eq:HIC}
 \es
In the special case where $D=\R^n$, we will denote the Gaussian heat kernel on the whole space by $G_n(x,y;t)$, as given by (\ref{eq:Kn}).

 In this section we provide a proof of a particular case of a result due to van den Berg \cite{vdBerg90}, which shows that away from the boundary the Dirichlet heat kernel is bounded below by a multiple of the Gaussian kernel for the heat equation on the whole space. In this context the result for $\Omega\subset\R^n$ is an easy corollary of the result in $\R$; in the one-dimensional case our proof significantly simplifies that of \cite{vdBerg90}.

\begin{theorem}\label{ourvdB}
  Let $\Omega$ be a domain in $\R^n$, and denote by $K_\Omega(x,y;t)$ the Dirichlet heat kernel on $\Omega$. Suppose that
  \begin{equation}\label{edef}
  \epsilon:=\inf_{z\in[x,y]}\ \dist(z,\partial\Omega)>0,
  \end{equation}
  where $[x,y]$ denotes the line segment joining $x$ and $y$ (so in particular $[x,y]$ is contained in the interior of $\Omega$). Then for $t\le \epsilon^2/n$
  $$
  K_\Omega(x,y;t)\ge \B^nG_n(x,y;t),
  $$
  where $\B =1-2/\e>0$.
\end{theorem}

Note that if $\Omega$ is convex then $\epsilon$ in (\ref{edef}) is simply given by
$$
\epsilon=\min(\dist(x,\partial\Omega),\dist(y,\partial\Omega)).
$$

We delay the proof of Theorem~\ref{ourvdB} for a moment. Following \cite{vdBerg90} we begin with the corresponding result for an interval in $\R$. Our proof is somewhat simpler than that of Lemma 8 in \cite{vdBerg90}, and non-probabilistic (cf.\ \cite{Varadhan}), since we are able to write down directly the Dirichlet kernel in terms of a sum of Gaussian kernels on the whole line.

We write $K_a$ for the one-dimensional heat kernel on $(-a,a)$.

\begin{lemma}\label{1dvdB}
  Take $a>0$. Then for any $x,y\in\Omega=(-a,a)$
 $$
  K_a(x,y;t)\ge G_1(x,y;t)\left[1-2\e^{-\epsilon^2/t}\right],
 $$
  where $\epsilon=\dist([x,y],\partial\Omega)$. In particular for $t\le\epsilon^2$
   \begin{equation}\label{rough1}
  K_a(x,y;t)\ge \B G_1(x,y;t),
  \end{equation}
  where $\B =1-2/\e>0$.
\end{lemma}

Note that Corollary \ref{good1d} in the Appendix improves the lower bound in (\ref{rough1}) to $\e^{-\pi^2t/4\epsilon^2}G_1(x,y;t)$ for all $t>0$, but the result of this lemma is sufficient for the arguments in the main body of this paper.

\begin{figure}[th]\label{figex}
\begin{centering}
  \psfrag{f}[][1]{$u_0$}
  \psfrag{0}[][1]{$0$}
  \psfrag{a}[][1]{$2a$}
  \psfrag{y}[][1]{$y$}
  \includegraphics[scale=.8]{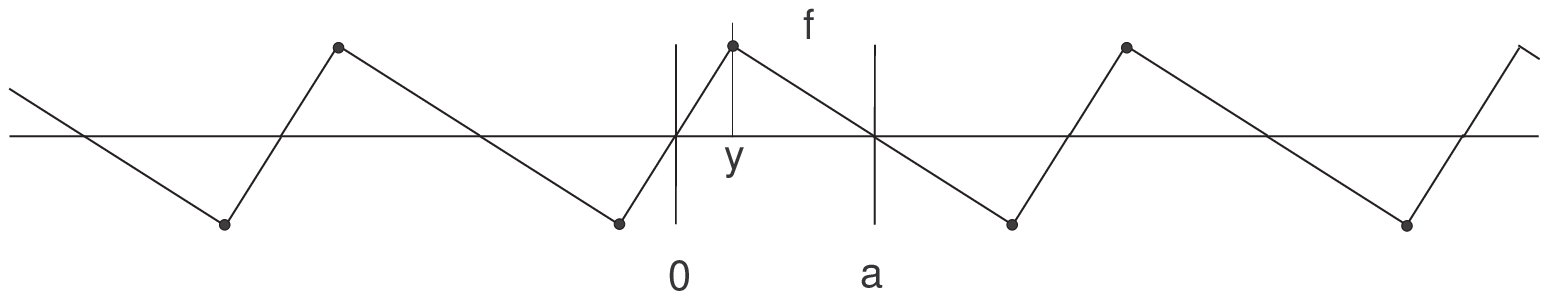}
  \end{centering}
  \begin{caption}{For a particular $u_0$ defined on $[0,2a]$, an illustration of the periodic extension that is anti-symmetric about $x=0$ and $x=2a$. Dots indicate positions and signs of the delta functions that give rise to $K_{(0,2a)}(x,y;t)$.}
  \end{caption}
\end{figure}

\begin{proof}
For notational reasons it is simpler to treat the problem on $(0,2a)$ rather than $(-a,a)$, but since the equations and the resulting lower bound are translation invariant this does not effect the result. We write down the Dirichlet heat kernel on $(0,2a)$ by reflection. The essential idea is shown in Figure 1: the action of the heat equation on $[0,2a]$ with initial data $u_0$ is the same as the action of the heat equation on $\R$ with the periodically extended initial data as illustrated, since this extension is antisymmetric about $0$ and $a$ and the Gaussian kernel $G_1(x,y;t)$ is symmetric about $x$ for any $x\in\R$.

The  contribution to the heat kernel for $x\in[0,2a]$ from a source at $y\in(0,2a)$ will be the sum of the Gaussian kernels with positive point sources at $y+4ka$ and negative point sources at $-y+4ka$ (see Figure 1, again), yielding
\begin{equation}\label{byi}
K_{(0,2a)}(x,y;t)=\frac{1}{\sqrt{4\pi t}}\sum_{k\in\Z}\e^{-|x-(y+4ka)|^2/4t}-\e^{-|x-(-y+4ka)|^2/4t},
\end{equation}
see Figure 2. [Even if one has doubts about the above derivation, it is clear that $K(x,y,t)$ in (\ref{byi}) satisfies the heat equation, $K(0,y;t)=K(2a,y;t)=0$, and $K(x,y;0)=\delta(y)$ for $x,y\in(0,2a)$.]

\begin{figure}[th]\label{figk}
\centering
\begin{tabular}{cccc}
   \includegraphics[scale=.2]{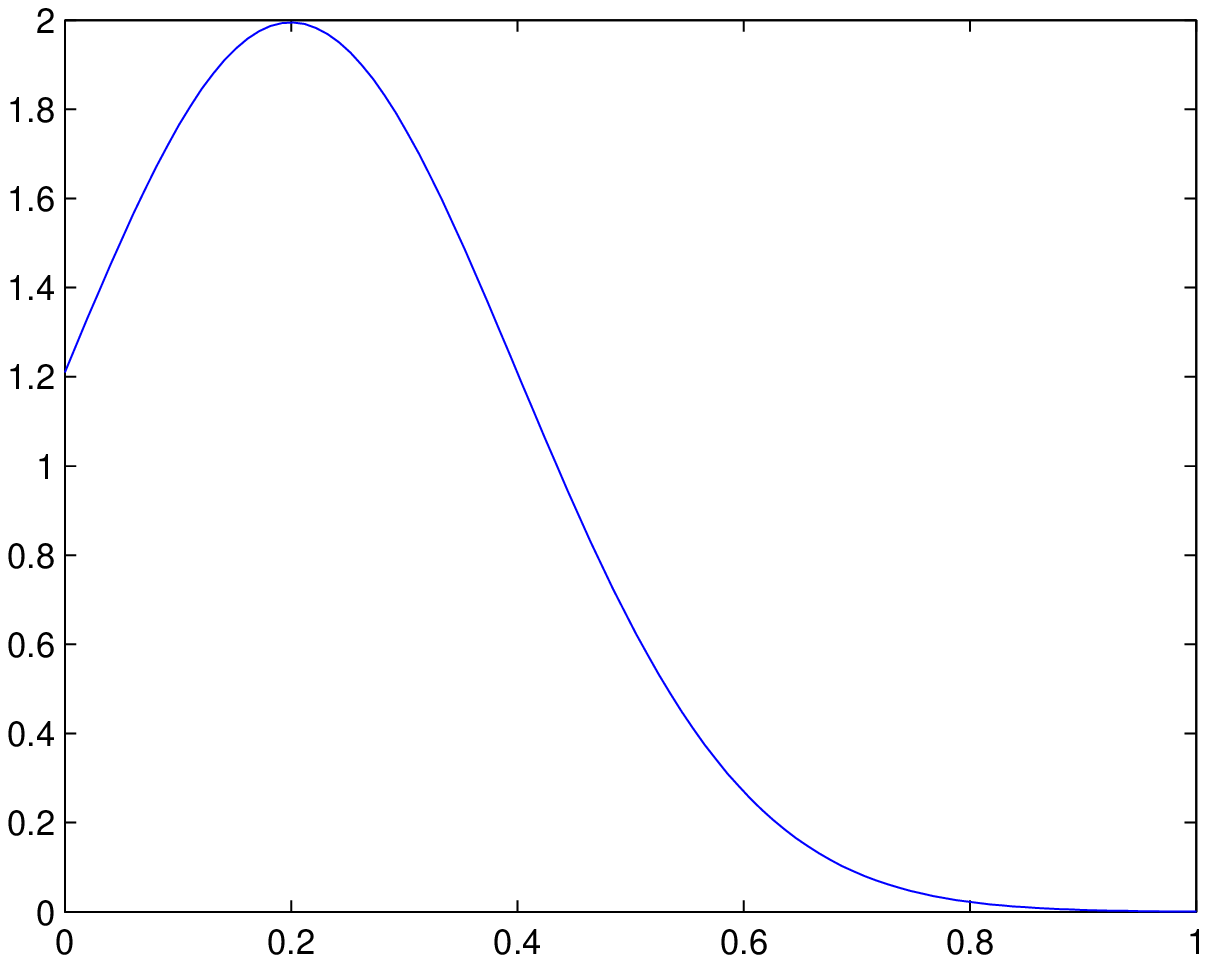}&
   \includegraphics[scale=.2]{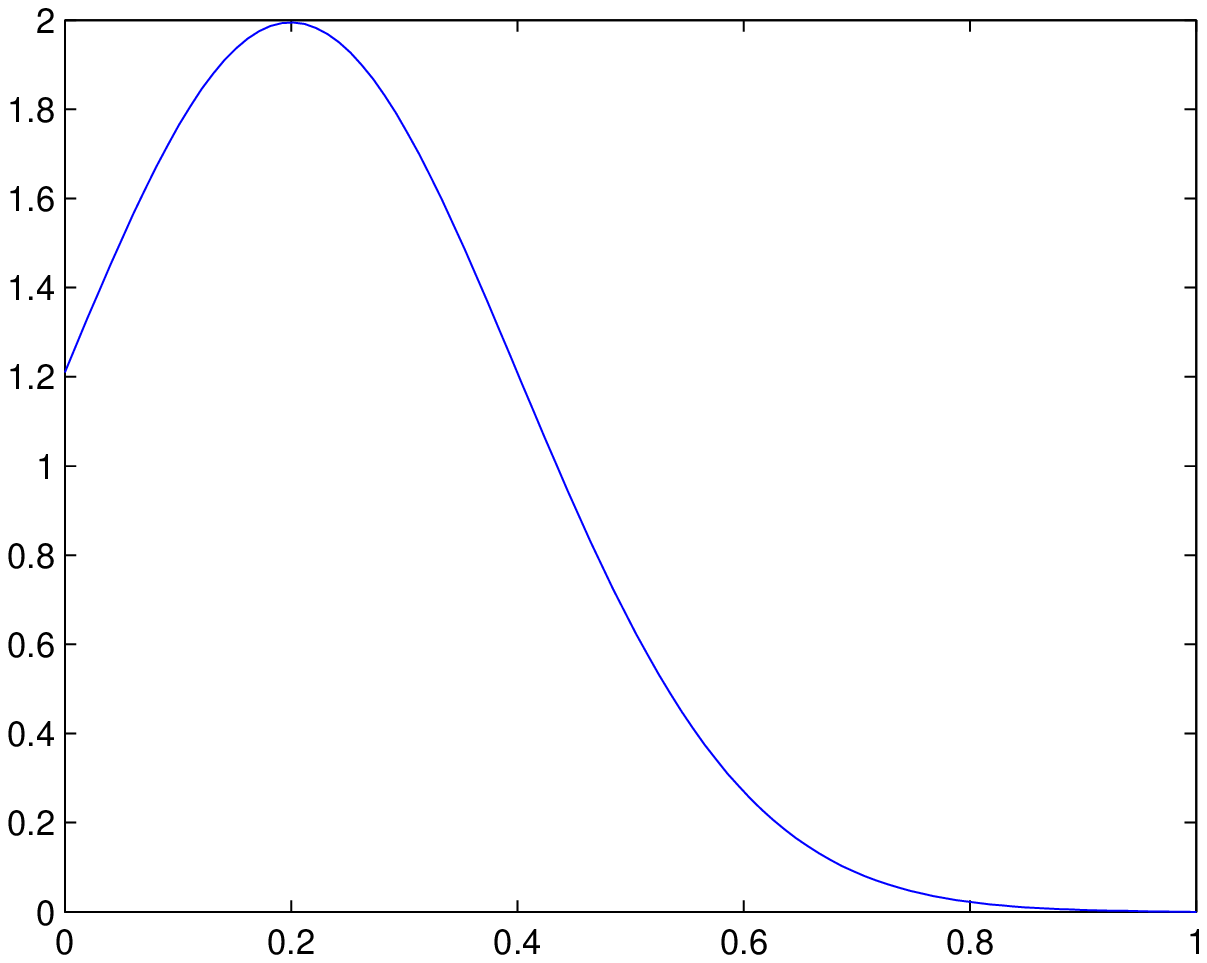}&
   \includegraphics[scale=.2]{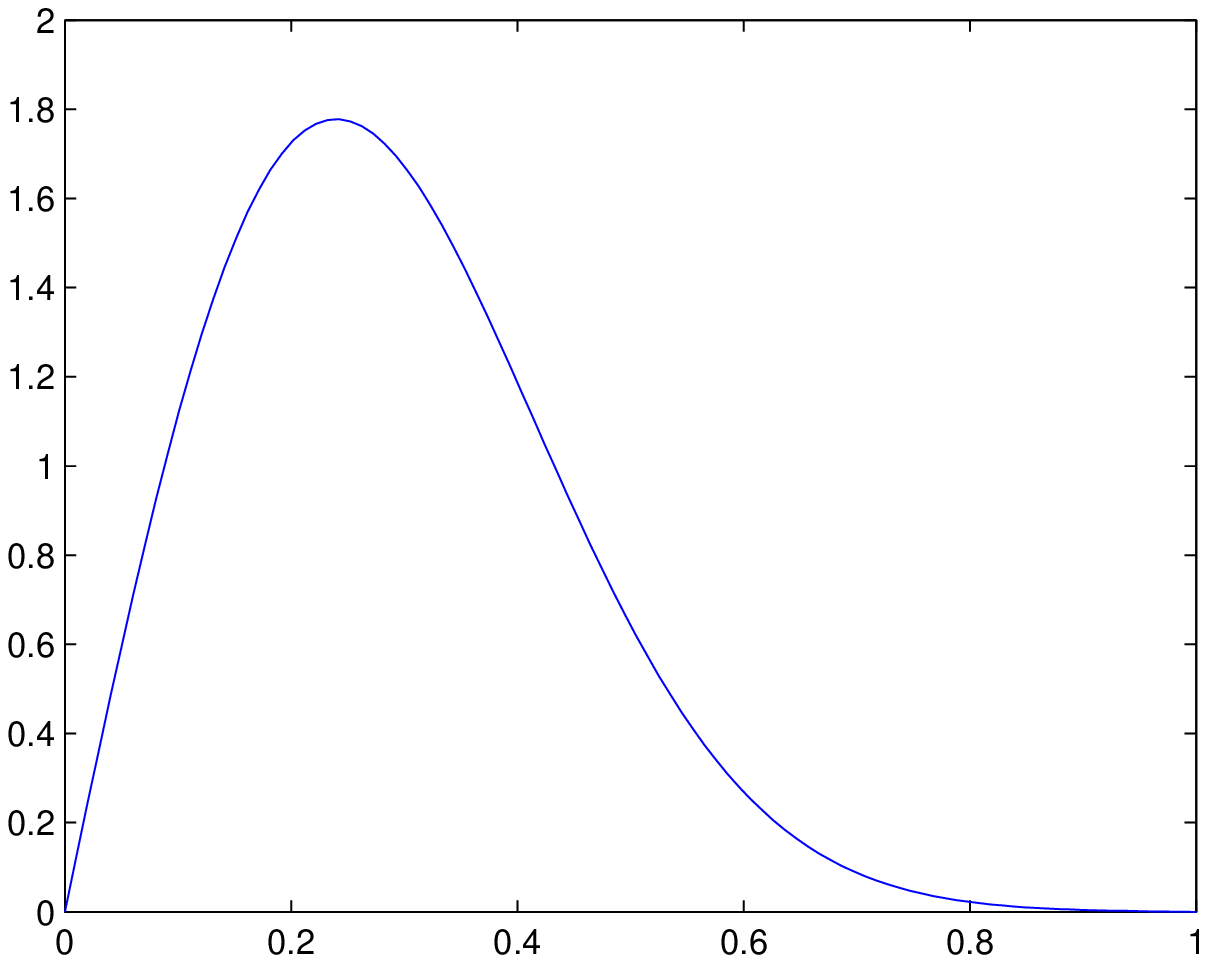}&
   \includegraphics[scale=.2]{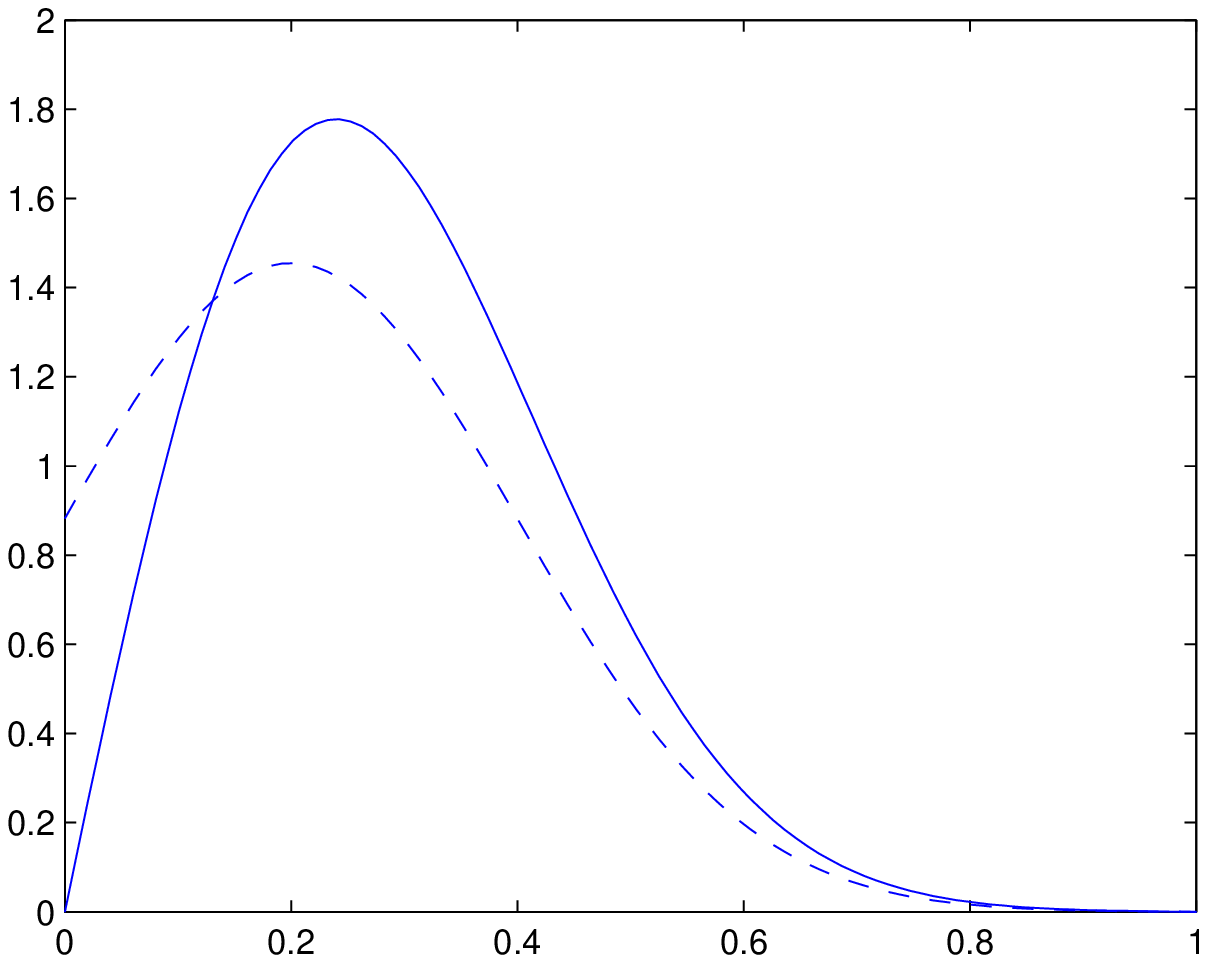}
   \end{tabular}
  \begin{caption}{Dirichlet heat kernel on $[0,1]$ as a sum of Gaussians for $y=0.2$, $t=0.02$. From left to right: Gaussian kernel on $\R$; one subtraction ($k=1$) to enforce boundary condition at $x=1$ (little change); second subtraction ($k=0$) towards satisfying the boundary condition at $x=0$; the heat kernel on $[0,1]$ (additional terms make little difference), with the lower bound from Lemma \ref{1dvdB} indicated by a dashed line.}
  \end{caption}
\end{figure}

Now we simply rewrite the sum:
\begin{align*}
&\sqrt{4\pi t}K_{(0,2a)}(x,y;t)=\sum_{k\in\Z}\e^{-|x-(y+4ka)|^2/4t}-\e^{-|x-(-y+4ka)|^2/4t}\\
&\ =\e^{-|x-y|^2/4t}-\e^{-|x+y|^2/4t}-\e^{-|x+y-4a|^2/4t}\\
&\quad+\sum_{k=1}^\infty\left\{\e^{-|x-(y+4ka)|^2/4t}+\e^{-|x-(y-4ka)|^2/4t}-\e^{-|(x-(-y+4a(k+1))|^2/4t}\right.\\
&\qquad\qquad\left.-\e^{-|(x-(-y-4ka)|^2/4t}\right\}\\
&\ =\e^{-|x-y|^2/4t}\left[1-\e^{-xy/t}-\e^{-(2a-x)(2a-y)/t}\right]\\
&\quad +\sum_{k=1}^\infty\e^{-|x-y-4ka|^2/4t}+\e^{-|x-y+4ka|^2/4t}-\e^{-|x+y-4(k+1)a|^2/4t}-\e^{-|x+y+2ka|^2/4t}.
\end{align*}
Noting that
$$|x+y+4ka|>|x-y+4ka|\qquad\mbox{and}\qquad|4(k+1)a-(x+y)|>|4ka-(x-y)|$$
for $k\ge 1$ and $x,y\in (0,2a)$, it follows that
$$
\sqrt{4\pi t}K_{(0,2a)}(x,y;t)\ge\e^{-|x-y|^2/4t}\left[1-2\e^{-\epsilon^2/t}\right].
$$
 Finally note that for $t\le\epsilon^2$ the term in the square brackets is at least $\B =1-2/\e$.\end{proof}

The idea of the proof of Theorem \ref{ourvdB}, inspired by that of Lemma 9 in \cite{vdBerg90}, is illustrated in Figure 3. We bound the Dirichlet heat kernel on $\Omega$ below by the kernel on the parallelepiped $\Pi$, which is simply the product of one-dimensional kernels which we can control using Lemma~\ref{1dvdB}. In this way the proof uses the monotonicity of the Dirichlet heat kernel with respect to the domain:
$$
\Omega\subset U\qquad\Rightarrow\qquad K_\Omega(x,y;t)\le K_U(x,y;t).
$$
A probabilistic proof can be found in \cite{vdBerg90}; an analytic proof using the theory of semigroups can be found in the notes by Arendt \cite{Arendt}.


\begin{figure}[th]\label{figcon}
\begin{centering}
  \psfrag{ex}[][][1]{$x$}
  \psfrag{y}[][][1]{$y$}
  \psfrag{dO}[][1]{$\partial\Omega$}
  \psfrag{e}[][1]{$\epsilon$}
  \psfrag{en}[][1]{$2\epsilon/\sqrt{n}$}
  \psfrag{P}[][1]{$\Pi$}
  \psfrag{O}[][1]{$\Omega$}
  \includegraphics[scale=1.2]{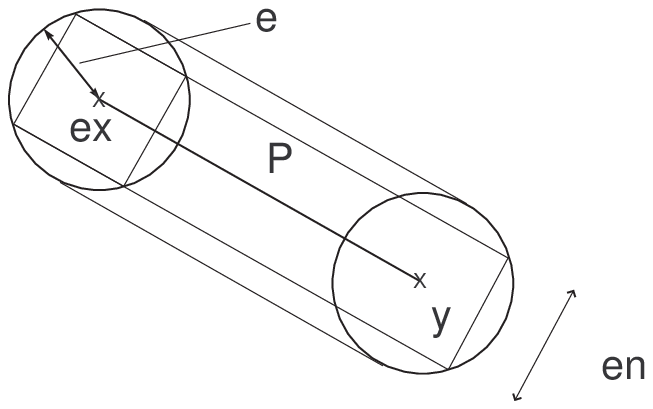}
  \begin{caption}
  {The parallelepiped $\Pi$ with $n-1$ sides of length $2\epsilon/\sqrt{n}$ when ${\rm dist}(x,\partial\Omega)=\epsilon$.}
  \end{caption}
  \end{centering}
\end{figure}

\begin{proof}[Proof of Theorem \ref{ourvdB}]

By the definition of $\epsilon$, the line segment joining $x$ and $y$ is entirely contained in a parallelepiped $\Pi$ that lies entirely within $\bar\Omega$, with one side of length $|x-y|+2\epsilon/\sqrt n$ and $n-1$ sides of length $2\epsilon/\sqrt n$, see Figure 3. Note that $x$ and $y$ are at least a distance $\epsilon/\sqrt n$ from all faces of $\Pi$. By monotonicity of the Dirichlet heat kernel with respect to the domain
$$
K_\Omega(x,y,t)\ge K_\Pi(x,y,t).
$$
If we now refer points in $\Pi$ to coordinate axes aligned along $[x,y]$ and in the perpendicular directions, so that $x=(\tilde x,0,\ldots,0)$ and $y=(\tilde y,0,\ldots,0)$, we can use the separation of variables property to write
\begin{align*}
K_\Pi(x,y;t)&=K_{\frac{1}{2}|x-y|+\epsilon/\sqrt n}(\tilde x,\tilde y,t)[K_{\epsilon/\sqrt n}(0,0,t)]^{n-1}\\
&\ge\beta(4\pi t)^{-1/2}\e^{-|\tilde x-\tilde y|^2/4t}\beta^{n-1}(4\pi t)^{-(n-1)/2}\\
&=(4\pi t)^{-n/2}\e^{-|x-y|^2/4t}\B^n\\
&=\B ^nG_n(x,y,t),
\end{align*}
for all $t\le\epsilon^2/n$, using the one-dimensional lower bound from Lemma \ref{1dvdB}.
\end{proof}

%
%
%

\section{A lower bound for the heat equation}
\label{main result}

Without loss of generality we henceforth assume that $\Om$ contains the origin. For $r>0$, $B(x,r)$ will denote the Euclidean ball in $\R^n $ of radius $r$ centred at $x$, and in an abuse of notation we write $B(r)$ for $B(0,r)$.

As an ingredient in the proof of our main result, we want to show that the action of the heat equation on the singular initial condition
\begin{equation}\label{w0is}
w_0(x)=|x|^{-\alpha}\chi_R:=\left\{\begin{array}{ll}
|x|^{-\A},& |x|\le R,\\
0,& |x|>R
\end{array}\right.
\end{equation}
does not have too pronounced an effect for short times. It is easy to see that
$$
w_0(x)>\phi\qquad\mbox{for}\quad |x|<\phi^{-1/\alpha};
$$
we now show that such a lower bound holds for a similar set of $x$ for sufficiently small times.

\begin{proposition}\label{prop:ball}
Fix $\A\in (0,n)$ and pick $R>0$  such that $\overline{B(R)}\subset \Om$. If $w_{\Omega}(x,t)$  denotes the solution of the linear heat equation on $\Om$ with initial condition\footnote{Strictly speaking $w_0$ is defined on the whole of $\R^n$; we take as initial condition the function in (\ref{w0is}) restricted to $\Omega$.} $w_0=|x|^{-\alpha}\chi_R$, as represented by (\ref{eq:DHK}),
then there exist  constants  $\sigma=\sigma(R, \A ,n)>0$ and $\phi_*=\phi_* (R, \A ,n)>0$ such that
\be\label{wbig}
w_{\Omega}(x,t)\ge\phi\qquad\mbox{for all}\quad |x|\le \sigma\phi^{-1/\alpha}\quad\mbox{and}\quad  0\le t\le \sigma\phi^{-2/\alpha}
\ee
for any $\phi > \phi_*$.
\end{proposition}

\begin{proof}
Let  $w$ denote the solution of the linear heat equation on  $\R^n$  with the same initial condition $w_0=|x|^{-\alpha}\chi_R$.
Let $\epsilon=\inf_{x\in B(R)}\dist(x,\partial\Omega)>0$; then it follows from   Theorem~\ref{ourvdB} that with $T=\epsilon^2/n$ we have
$$
K_\Omega(x,y,t)\ge \beta^nG_n(x,y,t),\qquad \forall x,y\in B(R),\quad t\in (0,T].
$$
From here on $c$ will denote any generic constant, and may change from line to line.

Taking  $|\xhat | =R$, $ t\in (0, T]$ and $\psi> 1$, we have
\begin{align*}
  w_{\Omega}(\hat x/\psi,t)&=\int_{\Omega} K_\Omega(\hat x/\psi,y,t)w_0(y)\,\d y =\int_{B(R)} K_\Omega(\hat x/\psi,y,t)|y|^{-\alpha}\,\d y\\
  &\ge c\int_{B(R)} G_n(\hat x/\psi,y,t)|y|^{-\alpha}\,\d y\\
  &\ge c(4\pi t)^{-n/2}\int_{B(R)}\e^{-|(\hat x/\psi) -y|^2/4t}|y|^{-\alpha}\,\d y\\
  &=c(4\pi t)^{-n/2}\int_{B(R)}\e^{-|\hat x-\psi y|^2/4\psi^2t}|y|^{-\alpha}\,\d y\\
  &= c\psi^\alpha(4\pi \psi^2 t)^{-n/2}\int_{B(\psi R)}\e^{-|\hat x-z|^2/4\psi^2t}|z|^{-\alpha}\,\d z\ge c\psi^\alpha w(\hat x,\psi^2t).
\end{align*}
Defining $M=M(R,\A ,n)>0$ by
\be
M=\inf\{ w(x,t)\ :\ |x|=R,\ 0\le t\le T\}
\label{eq:M}\ee
it follows that $w_{\Omega}(x,t)\ge cM\psi^\alpha$ for all $|x|=R\psi^{-1}$ and $t\in (0,\psi^{-2}T]$.
Furthermore, $w_{\Omega}(x,0)=|x|^{-\A}\ge R^{-\A}\psi^\alpha$ for all $|x|\le R\psi^{-1}$.
Consequently, by the  parabolic maximum principle,
\bs
w_{\Omega}(x,t)\ge \phi_*\psi^\alpha\qquad\mbox{for all}\quad |x|\le R\psi^{-1}\quad\mbox{and}\quad  0\le t\le \psi^{-2}T,
\es
where $\phi_*:=\min(cM,R^{-\A})>0$. With  $\sigma=\min(R\phi_*^{1/\alpha},T\phi_*^{2/\alpha} )>0$ and $\phi = \phi_*\psi^\alpha > \phi_*$, one therefore obtains
\[
w_{\Omega}(x,t)\ge\phi\qquad\mbox{for all}\quad  |x|\le \sigma\phi^{-1/\alpha}\quad\mbox{and}\quad   0\le t\le \sigma\phi^{-2/\alpha}.\qedhere
\]
\end{proof}

\section{Non-existence of local solutions}

In this section we prove the non-existence of local solutions, taking the following as our (essentially minimal) definition of such a solution. Note that the non-existence of a solution in the sense of Definition~\ref{def:soln} implies the non-existence of mild solutions and of classical solutions \cite[p.\ 77--78]{SQ}.

\begin{definition}{\cite[p.\ 78]{SQ}}
Given $f\ge 0$  and $u_0\ge0$ we say that $u$ is a {\rm{local integral solution}} of {\rm(P)} on $[0,T)$ if $u:\Om \times[0,T)\to[0,\infty]$ is measurable, finite almost everywhere, and satisfies
\be
u(t)=S_{\Om}(t)u_0+\int_0^t S_{\Om}(t-s)f(u(s))\, \d s\label{eq:mild}
\ee
 almost everywhere in $\Om \times[0,T)$.\label{def:soln}
\end{definition}

We now prove our main result, in which we obtain instantaneous blow-up in $L^1_{\rm loc}(\Om )$ for certain initial data in $L^q(\Om )$, $1\le q<\infty$, under the asymptotic growth condition (\ref{limsup}) when $f$ is non-decreasing.

\begin{theorem}\label{thm:main}
 Let $q\in [1,\infty )$. Suppose that $f:[0,\infty )\to [0,\infty )$ is non-decreasing. If
\begin{equation}\label{limsup}
\limsup_{s\to\infty}s^{-\gamma}f(s)=\infty
\end{equation}
for some $\gamma>q(1+\frac{2}{n})$, then there exists $u_0\in L^q(\Om )$ such that {\rm(P)} possesses no local integral solution. Indeed, any solution $u(t)$ that satisfies (\ref{eq:mild}) is not in $L^1_{\rm loc}(\Om )$ for any $t>0$.
\label{thm:nonexist}
\end{theorem}

\begin{proof}
We show that for small $t>0$, $u(t)\notin L^1_{\rm loc}(\Omega)$ and hence, arguing as in \cite[Theorem 4.1]{LRS1},  there can be no local integral solution of (P).

Choose $B(R)$ as in Proposition~\ref{prop:ball}. Set $\alpha=(n+2)/\gamma<n/q$, so that
$$
 \limsup_{s\to\infty}s^{-(n+2)/\alpha}f(s)=\infty.
 $$
 Then in particular there exists a sequence $\phi_i\to\infty$ such that
 $$
 f(\phi_i)\phi_i^{-(n+2)/\alpha}\to\infty\qquad\mbox{as}\quad i\to\infty.
 $$

Now take $u_0=|x|^{-\alpha}\chi_R\in L^q(\Om )$.  Defining $T$ as in the proof of Proposition~\ref{prop:ball}, fix $t<\min(T,1)$ and choose $i$ sufficiently large such that $\phi_i> \phi_*$,
 $\sigma\phi_i^{-1/\alpha}<R/2$ and $\sigma\phi_i^{-2/\alpha}\le t$.
 Clearly, by comparison, $u\ge w_{\Om}\ge 0$. Hence by monotonicity of $f$ and Theorem~\ref{ourvdB},
\begin{align*}
I &:=\int_{B(R)}u(t)\,\d x\ge  \int_{B(R)}\int_0^t[S_{\Om}(t-s)f(w_{\Om}(\cdot,s))](x)\,\d s\,\d x\\
&=\int_0^{t}\int_{B(R)}\int_{\Om} K_{\Om}(t-s,x,y)f(w_{\Om}(y,s))\,\d y\,\d x\,\d s\\
&\ge  c\int_0^{\sigma\phi_i^{-2/\alpha}}\int_{B(R)}\int_{B(\sigma\phi_i^{-1/\alpha})}G_n(t-s,x,y)f(\phi_i)\,\d y\,\d x\,\d s\\
&=  cf(\phi_i)\int_0^{\sigma\phi_i^{-2/\alpha}}\int_{B(\sigma\phi_i^{-1/\alpha})}(4\pi(t-s))^{-n/2}\int_{B(R)}\e^{-|x-y|^2/4(t-s)}\,\d x\,\d y\,\d s.
\end{align*}
Let $z=x-y$. Since  $|y|\le \sigma\phi_i^{-1/\alpha}<R/2$, it follows that
$$
\{z=x-y : x\in B(R)\}\supset B(R/2).
$$
Thus
\begin{align*}
I&\ge  cf(\phi_i)\int_0^{\sigma\phi_i^{-2/\alpha}}\int_{B\left( \sigma\phi_i^{-1/\alpha}\right)}(4\pi(t-s))^{-n/2}\int_{B(R/2)}\e^{-|z|^2/4(t-s)}\,\d z\,\d y\,\d s\\
&=  cf(\phi_i)\int_0^{\sigma\phi_i^{-2/\alpha}}\int_{B\left( \sigma\phi_i^{-1/\alpha}\right)}\int_{B\left( R/\sqrt{t-s}\right)}\e^{-|v|^2}\,\d v\,\d y\, \d s\quad \left( v=z/2\sqrt{t-s}\right)\\
&\ge
cf(\phi_i)\int_0^{\sigma\phi_i^{-2/\alpha}}\int_{B\left( \sigma\phi_i^{-1/\alpha}\right)}\int_{B(R)}\e^{-|v|^2}\,\d v\,\d y\,\d s\quad (\sqrt{t-s}\le 1)\\
&= cf(\phi_i)\int_0^{\sigma\phi_i^{-2/\alpha}}(\sigma\phi_i^{-1/\alpha})^n\,\d s \\
&= cf(\phi_i)\phi_i^{-(n+2)/\alpha}
\to  \infty \ {\rm as}\  i\to\infty.\qedhere
\end{align*}
\end{proof}

Note that this result also guarantees instantaneous blow-up of solutions of
$$
u_t=\Delta u+g(u)
$$
for any $g$ such that $g(s)\ge f(s)$, where $f$ satisfies the conditions of Theorem \ref{thm:main}, even if $g$ is not monotonic. In particular, for the canonical Fujita equation
\begin{equation}\label{u2p}
u_t=\Delta u +u^p,
\end{equation}
our argument shows the non-existence of local solutions when $p>q(1+\frac{2}{n})$. The sharp result in this case is known to be $p>1+\frac{2}{n}q$ \cite{WEISSLER1980,WEISSLER1986} with equality allowed if $q=1$ \cite{CZ}.

The existence of a finite limit in (\ref{limsup}) implies that $f(s)\le c(1+s^\gamma)$, and hence by comparison with (\ref{u2p}) is sufficient for the local existence of solutions provided that $\gamma<1+\frac{2}{n}q$ \cite{WEISSLER1979}. We currently, therefore, have an indeterminate range of $\gamma$,
 $$
 1+\frac{2}{n}q\le\gamma\le q(1+\frac{2}{n})
 $$
 for which we do not know whether (\ref{limsup}) characterises the existence or non-existence of local solutions.

\section{A very `bad' Osgood $f$}

To finish, using a variant of the construction in \cite{LRS1}, we provide an example of an $f$ that satisfies the Osgood condition (\ref{Osgood1}) but for which
\begin{equation}\label{bad}
\limsup_{s\to\infty}s^{-\gamma}f(s)=\infty,\qquad\mbox{for every}\quad \gamma\ge0.
\end{equation}

\begin{theorem}\label{cor:lip}
  There exists a locally Lipschitz function $f:[0,\infty)\to[0,\infty)$ such that $f(0)=0$, $f$ is non-decreasing, and $f$ satisfies the Osgood condition
  $$
  \int_1^\infty\frac{1}{f(s)}\,\d s=\infty,
  $$
  but nevertheless (\ref{bad}) holds. Consequently, for this $f$, for any $1\le q<\infty$ there exists a $u_0\in L^q(\Omega)$ such that {\rm(P)} has no local integral solution.
\end{theorem}

\begin{proof}
Fix $\phi_0=1$ and define inductively the sequence $\phi_i$  via
$$
\phi_{i+1}=\e^{\phi_i}.
$$
Clearly, $\phi_i\to\infty$ as $i\to\infty$. Now define $f:[0,\infty )\to [0,\infty )$ by
\ba
f(s)=\left\{
\begin{array}{ll}
(\e-1)s,& s\in J_0:=[0,1],\\
\phi_{i}-{\phi_{i-1}},& s\in I_i:=[\phi_{i-1},{\phi_{i}}/2],\ i\ge 1,\\
\ell_i(s),&s\in J_i:=(\phi_i/2,\phi_i),\ i\ge 1,
\end{array}\right.
\label{eq:f}
\ea
where $\ell_i$ interpolates linearly between the values of $f$ at $\phi_{i}/2$ and $\phi_i$. By construction $f(0)=0$, $f$ is non-decreasing, and $f$ is Osgood since
$$
\int_1^\infty \frac{1}{f(s)}\,\d s\ge\sum_{i=1}^\infty\int_{I_i}\frac{1}{f(s)}\,\d s= \sum_{i=1}^\infty\frac{\phi_i/2-\phi_{i-1}}{\phi_i-\phi_{i-1}}=+\infty.
$$
However, $f(\phi_i)=\e^{\phi_i}-\phi_i$, and so for any $\gamma\ge0$
$$
\lim_{i\to\infty}\phi_i^{-\gamma}f(\phi_i)\to\infty\qquad\mbox{as}\quad i\to\infty,
$$
which shows that (\ref{bad}) holds.\end{proof}

This example shows that there exist semilinear heat equations that are globally well-posed in $L^{\infty}(\Om )$, yet ill-posed in every $L^{q}(\Om )$
for $1\le q<\infty$.

\section{Appendix: Gaussian lower bound on the heat kernel for all $t>0$}

For the sake of completeness we now follow \cite{vdBerg90} (see also \cite{vdBerg92}) and use the result of Lemma \ref{1dvdB} to obtain a lower bound on\footnote{Recall that we use the notation $K_a(x,y;t)$ for the one-dimensional heat kernel on $(-a,a)$.}  $K_a(x,y;t)$ in terms of $K_\epsilon(0,0;t)$. We then bound $K_\epsilon(0,0;t)$ below by supplementing the bound from Lemma \ref{1dvdB} with information from the eigenfunction expansion of the kernel. This will allow us a simple proof of a similar form of lower bound on a general domain $\Omega$ when $[x,y]\subset\Omega$.

The main idea in the proof is to use repeatedly the semigroup property of the heat semigroup in the form
$$
K_a(x,y;t)=\int_{(-a,a)}K_a(x,z;t)K_a(z,y;t)\,\d z.
$$

\begin{proposition}\label{vdbprop}
 The one-dimensional heat kernel on $\Omega=(-a,a)$ satisfies
  \begin{equation}\label{KbdK}
  K_a(x,y;t)\ge \e^{-|x-y|^2/4t}K_\epsilon(0,0,t)
  \end{equation}
  for all $x,y\in(-a,a)$ and $t>0$, where $\epsilon=\dist([x,y],\partial\Omega)$.
\end{proposition}

\begin{proof}
Take $x,y\in(-a,a)$, $t>0$, and $m\in\N$ with $m$ sufficiently large that $1-2\e^{-m\epsilon^2/t}>0$. For $j=0,1,\ldots,m$ set $x_j=x+jz$, where $z=(y-x)/m$. Then using the semigroup property
\begin{align*}
&K_a(x,y;t)\\
&\ =\int_{\Omega^{m-1}} K_a(x,y_1;t/m)\left\{\prod_{j=1}^{m-2}K_a(y_j,y_{j+1};t/m)\right\}K_a(y_{m-1},y;t/m)\,\d^{m-1}y,
\end{align*}
writing $\d^{m-1}y$ for $\d y_1\cdots\,\d y_{m-1}$.

Now note that  $B(x_j,\epsilon)\subset\Omega$ for every $j=0,\ldots,m$, and so
$$
K_a(x,y;t)
\ge\int_{B(\epsilon)^{m-1}}\prod_{j=0}^{m-1}K_a(x_j+w_j,x_{j+1}+w_{j+1};t/m)
  \,\d^{m-1}w,
$$
setting $w_0=w_m=0$ and $w_j=y_j-x_j$ for $j=1,\ldots,m-1$. Using Lemma \ref{1dvdB} we obtain the lower bound
\begin{align*}
K_a&\ge C_{m,t}\int_{B(\epsilon)^{m-1}}\prod_{j=0}^{m-1}G_1(x_j+w_j,x_{j+1}+w_{j+1};t/m)\,\d^{m-1}w\\
&=C_{m,t}\int_{B(\epsilon)^{m-1}}(4\pi t/m)^{-m/2}\exp\left(-\frac{m\sum_{j=0}^{m-1}|z+w_{j+1}-w_j|^2}{4t}\right)\,\d^{m-1}w,
\end{align*}
where $C_{m,t}=[1-2\e^{-m\epsilon^2/t}]^m$.

Elementary algebra gives
$$
m\sum_{j=0}^{m-1}|z+w_{j+1}-w_j|^2
=|x-y|^2+m\sum_{j=0}^{m-1}|w_{j+1}-w_j|^2,
$$
and therefore
\begin{align*}
K_a
&\ge C_{m,t}\e^{-|x-y|^2/4t}\int_{B(\epsilon)^{m-1}}(4\pi t/m)^{-m/2}\exp\left(-\frac{m\sum_{j=0}^{m-1}|w_{j+1}-w_j|^2}{4t}\right)\,\d^{m-1}w\\
&=C_{m,t}\e^{-|x-y|^2/4t}\int_{B(\epsilon)^{m-1}}\prod_{j=0}^{m-1}G_1(w_j,w_{j+1};t/m)\,\d^{m-1}w.
\end{align*}
Now we can use monotonicity of the heat kernel, $G_1\ge K_\epsilon$, to obtain
\begin{align*}
K_a(x,y;t)&\ge C_{m,t}\e^{-|x-y|^2/4t}\int_{B(\epsilon)^{m-1}}\prod_{j=0}^{m-1}K_\epsilon(w_j,w_{j+1};t/m)\,\d^{m-1}w\\
&=C_{m,t}\e^{-|x-y|^2/4t}K_\epsilon(0,0,t),
\end{align*}
using the semigroup property of $K_\epsilon$ and recalling that $w_0=w_m=0$. Finally, noting that $C_{m,t}\to1$ as $m\to\infty$, we obtain (\ref{KbdK}).
\end{proof}

We now obtain a lower bound on $K_a(0,0;t)$ using the eigenfunction expansion of the kernel.

\begin{lemma}\label{allt}
For all $t>0$
\begin{equation}\label{lowerat}
K_a(0,0;t)\ge\frac{1}{\sqrt{4\pi t}}\e^{-\pi^2t/4a^2}.
\end{equation}
\end{lemma}

\begin{proof}
The eigenfunctions of $u_{xx}=\lambda u$ with $u(0)=u(2a)=0$ are $\sin k\pi x/2a$ with corresponding eigenvalues $-k^2\pi^2/4a^2$: the kernel is therefore
$$
K_{(0,2a)}(x,y;t)=\frac{1}{a}\sum_{k=1}^\infty \e^{-k^2\pi^2t/4a^2}\sin (k\pi x/2a)\sin (k\pi y/2a).
$$
Since $K_a(0,0;t)=K_{(0,2a)}(a,a;t)$ we obtain
\begin{align*}
K_a(0,0;t)&=\frac{1}{a}\sum_{k=1}^\infty \e^{-k^2\pi^2t/4a^2}\sin^2 (k\pi/2)\\
&=\frac{1}{a}\sum_{k=0}^\infty \e^{-(2k+1)^2\pi^2t/4a^2}\ge\frac{1}{a}\e^{-\pi^2t/4a^2},
\end{align*}
from which (\ref{lowerat}) follows for $a\le(4\pi t)^{1/2}$. For $t\le a^2/4\pi$, we use Lemma \ref{1dvdB} to give
$$
K_a(0,0;t)\ge 1-2\e^{-\epsilon^2/t};
$$
now simply observe that $\e^{-1/s} < s/4$ and $\e^{-s} < 1 - (s/2)$ for $0 < s < 1/3$,
and so certainly $1 - 2\e^{-1/s}\ge \e^{-s}$ for $0 < s \le 1/(4\pi)< 1/3$, and thus the bound in (\ref{lowerat}) holds for all $t>0$ as claimed.\end{proof}

%
%
%
%

Combining the results of Proposition \ref{vdbprop} and Lemma \ref{allt} finally yields the required lower bound in one dimension.

\begin{corollary}\label{good1d}
If $\Omega=(-a,a)$, $[x,y]\subset\Omega$, and $\epsilon=\dist([x,y],\partial\Omega)$ then
$$
K_a(x,y;t)\ge \e^{-\pi^2t/4\epsilon^2}G_1(x,y;t)\qquad\mbox{for all}\quad t>0.
$$
\end{corollary}

For $\Omega\subset\R^n$ the result follows using the argument in the proof of Theorem \ref{ourvdB}, in particular the inequality
$$
K_\Omega(x,y,t)\ge K_\Pi(x,y,t)\ge K_{\frac{1}{2}|x-y|+\epsilon/\sqrt n}(\tilde x,\tilde y,t)[K_{\epsilon/\sqrt n}(0,0,t)]^{n-1}.
$$

\begin{corollary}\label{goodnd}
If $\Omega\subset\R^n$, $[x,y]\subset\Omega$, and $\epsilon=\dist([x,y],\partial\Omega)$ then
$$
K_a(x,y;t)\ge \e^{-n^2\pi^2t/4\epsilon^2}G_n(x,y;t)\qquad\mbox{for all}\quad t>0.
$$
\end{corollary}

We note that the argument in \cite{vdBerg90} does not require the line segment $[x,y]$ to be contained in $\Omega$, leading to a lower bound that depends on the curvature of the geodesic joining $x$ and $y$.


\vspace{0.5cm}
\noindent{\bf Acknowledgements}
The authors acknowledge support under the following grants: EPSRC Leadership Fellowship  EP/G007470/1 (JCR);  EPSRC Doctoral Prize EP/P50578X/1 (MS).





%

\end{document}